\documentclass{llncs}

\usepackage[utf8]{inputenc}
\usepackage{enumerate}
\usepackage[shortlabels]{enumitem}
\usepackage{graphicx}
\usepackage{epstopdf}
\usepackage{amsfonts}
\usepackage{amssymb}
\usepackage{amsmath}
\usepackage[labelfont=bf]{caption}
\usepackage{sectsty}
\usepackage{wasysym}
\usepackage{hyperref}

\newcommand{\Sel}{\mathrm{Sel}}
\newcommand{\gen}{\mathrm{gen}}

\begin{document}

\title{Selection-based Approach\\ to Cooperative Interval Games}

\author{Jan Bok\inst{1} \and Milan Hlad\'{i}k\inst{2}}

\institute{
Computer Science Institute of Charles University, Malostransk\'{e} n\'{a}m\v{e}st\'{i} 25, 11800, Prague, Czech Republic \email{bok@iuuk.mff.cuni.cz}
\and
Department of Applied Mathematics, Faculty of Mathematics and Physics, Charles University, Malostransk\'{e} n\'{a}m\v{e}st\'{i} 25, 11800, Prague, Czech Republic \email{hladik@kam.mff.cuni.cz}
}

\date{}

\maketitle

\begin{abstract}
Cooperative interval games are a generalized model of cooperative games in
which the worth of every coalition corresponds to a closed interval representing
the possible outcomes of its cooperation. Selections are all possible
outcomes of the interval game with no additional uncertainty.

We introduce new selection-based classes of interval games and prove their
characterization theorems and relations to existing classes based on the interval weakly better
operator. We show new results regarding the core and imputations and examine a
problem of equivalence for two different versions of the core, the main stability
solution of cooperative games. Finally, we introduce the definition of strong imputation and strong core
as universal solution concepts of interval games.
\\
\\
This is a slightly updated version of our paper from ICORES 2015.
\end{abstract}

\section{Introduction}

Uncertainty and inaccurate data are issues occurring very often in real-world
situations. Therefore it is important to be able to make decisions even when
the exact data are not available and some bounds on them are known.

In classical cooperative game theory, every group of players
(\emph{coalition}) knows the precise reward for their cooperation; in
cooperative interval games, only the worst and the best possible outcome are
known. Such situations can be naturally modeled with intervals encapsulating
these outcomes. This model is especially useful if we have no additional
assumption on probability distribution on this interval. In some sense,
cooperative interval games get the best of both worlds. We count in all the
possible outcomes, yet our model is sufficiently simple to analyze.

Cooperation under interval uncertainty was first considered by Branzei,
Dimitrov and Tijs in 2003 to study bankruptcy situations \cite{bank} and later
further extensively studied by Alparslan G\"{o}k in her Ph.D. thesis
\cite{gokphd} and in other papers written together with Branzei et al. (see
references in \cite{coop1}).

However, their approach is almost exclusively aimed at interval solutions,
that is on payoff distributions consisting of intervals and thus containing
yet another uncertainty. This is in contrast with selections -- possible
outcomes of an interval game with no additional uncertainty. The selection-
based approach was never systematically studied and not very much is known.
This paper is trying to fix this and summarizes our results regarding the
selection-based approach to interval games.

The paper has the following structure. Section 2 is a preliminary section that
presents necessary definitions and facts on classical cooperative games,
interval analysis, and cooperative interval games. Section 3 is devoted to new
selection-based classes of interval games. We consequently prove their
characterizations and relations to existing classes. Section 4 focuses on the
so-called core incidence problem which asks under which conditions are the
selection core and the set of payoffs generated by the interval core equal. In
Section 5, the definitions of strong core and strong imputation are introduced
as new concepts. We show some simple results on the strong core, one of them
being a characterization of games with the strong imputation and strong core.
Finally, we conclude this paper with a summary of our results and possible
directions for future research.

\section{Preliminaries}

\subsection{On mathematical notation}
\begin{itemize}[leftmargin=*]
\item We will use $\le$ relation on real vectors. For every $x,y \in \mathbb{R}^N$
we write $x \le y$ if $x_i \le y_i$ holds for every $1 \le i \le N$.
\item We do not use symbol $\subset$ in this paper. Instead, $\subseteq$ and $\subsetneq$ are
used for subset and proper subset, respectively, to avoid ambiguity.
\end{itemize}

\subsection{Classical cooperative game theory}

Comprehensive sources on classical cooperative game theory are for example
\cite{branzei2008models,driessen1988cooperative,gilles2010cooperative,peleg2007introduction}. For more
information on applications, see e.g. \cite{bilbao2000cooperative,combopt,insurance}. Here we present only necessary background
theory for studying interval games. We examine games with transferable utility
(TU) and therefore by a cooperative game we mean a cooperative TU game.

\begin{definition}
  \emph{(Cooperative game)} 
  A cooperative game is an ordered pair $(N, v)$, where $N = \{1,2,\ldots ,n\}$
  is a set of players and $v: 2^N \to \mathbb{R}$ is a characteristic function
  of the cooperative game. We further assume that $v(\emptyset) = 0$.
\end{definition}

The set of all cooperative games with a player set $N$ is denoted by $G^N$.

Subsets of $N$ are called \emph{coalitions} and $N$ itself is called the
\emph{grand coalition}.

We often write $v$ instead of $(N,v)$, because we can easily identify a game
with its characteristic function without loss of generality.

To further analyze players' gains, we will need a \emph{payoff vector} which
can be interpreted as a proposed distribution of rewards between players.

\begin{definition}
  \emph{(Payoff vector)} A payoff vector for a cooperative game $(N, v)$ is a vector
  $x \in \mathbb{R}^N$ with $x_i$ denoting the reward given to the $i$th player.
\end{definition}

\begin{definition}
  \emph{(Imputation)} An imputation of $(N,v) \in G^N$ is a vector $x \in \mathbb{R}^N$
  such that $\sum_{i \in N} x_i = v(N)$ and $x_i \ge v(\{i\})$ for every $i \in N$.

  The set of all imputations of a given cooperative game $(N,v)$ is denoted by
  $I(v)$.
\end{definition}

\begin{definition}
  \emph{(Core)} The core of $(N,v) \in G^N$ is the set
  $$C( v ) = \Bigl\{ x \in I(v); \; \sum_{ i \in S } x_i \geq v(S), \forall S \subseteq N \Bigr\}.$$
\end{definition}

There are many important classes of cooperative games. Here we show the most
important ones.

\begin{definition} \emph{(Monotonic game)} A game $(N,v)$ is monotonic if for
every $T \subseteq S \subseteq N$ we have
$$v(T) \le v(S)\textrm{.}$$
\end{definition}

Informally, in monotonic games, bigger coalitions are stronger.

\begin{definition} \emph{(Superadditive game)} A game $(N,v)$ is superadditive if for
every $S,T \subseteq N$, $S \cap T = \emptyset$ we have
$$v(T) + v(S) \le v(S \cup T)\textrm{.}$$
\end{definition}

In a superadditive game, a coalition has no incentive to divide itself since
together they will always achieve at least as much as separated.

Superadditive games are not necessarily monotonic. Conversely, monotonic games
are not necessarily superadditive. However, these classes have a nonempty
intersection. Check Caulier's paper \cite{caulier2009note} for more details on
the relationship between these two classes.

\begin{definition} \label{add} \emph{(Additive game)} A game $(N,v)$ is additive if for
every $S,T \subseteq N$, $S \cap T = \emptyset$ we have
$$v(T) + v(S) = v(S \cup T)\textrm{.}$$
\end{definition}

Observe that additive games are superadditive as well.

Another important type of game is a \emph{convex game}.

\begin{definition} \emph{(Convex game)} A game $(N,v)$ is convex if its
characteristic function is supermodular. The characteristic function
is supermodular if for every $S \subseteq T \subseteq N$,
$$v(T) + v(S) \le v(S \cup T) + v(S \cap T)\textrm{.}$$
\end{definition}

Clearly, supermodularity implies superadditivity. 

Convex games have many nice properties. We remind the most important one.

\begin{theorem} (Shapley 1971 \cite{shapley1971cores})
  If a game $(N,v)$ is convex, then its core is nonempty.
\end{theorem}

\subsection{Interval analysis}

\begin{definition}
  \emph{(Interval)} An interval $X$ is a set
  $$X := [\underline{X},\overline{X}] =\{x \in \mathbb{R}: \underline{X} \le x \le \overline{X}\}\textrm{.}$$
  with $\underline{X}$ being the lower bound and $\overline{X}$ being the upper bound of the interval.
\end{definition}

From now on, by an interval, we mean a closed interval. The set of
all real intervals is denoted by $\mathbb{IR}$.

The following definition (from \cite{moore2009introduction}) shows how to do basic arithmetics with intervals.

\begin{definition}
  For every $X, Y, Z \in \mathbb{IR}$ and $0 \notin Z$ define
\begin{align*}
  X + Y &:= [\underline{X} + \underline{Y}, \overline{X} + \overline{Y}]\textrm{,}\\
  X - Y &:= [\underline{X} - \overline{Y}, \overline{X} - \underline{Y}]\textrm{,}\\
  X \cdot Y &:= [\min S , \max S],\ S = \{\underline{X}\overline{Y}, \overline{X}\underline{Y}, \underline{X}\underline{Y}, \overline{X}\overline{Y}\}\textrm{,}\\
  X\,/\,Z &:= [\min S , \max S],\ S = \{\underline{X}/\overline{Z}, \overline{X}/\underline{Z}, \underline{X}/\underline{Z}, \overline{X}/\overline{Z}\}\textrm{.}
\end{align*}
\end{definition}

\subsection{Cooperative interval games}
\label{coopgamesintro}

This section aims at presenting results on cooperative interval games necessary
to grasp our contribution to theory.

\begin{definition}
  \emph{(Cooperative interval game)} 
  A cooperative interval game is an ordered pair $(N, w)$, where $N = \{1,2,\ldots ,n\}$
  is a set of players and $w: 2^N \to \mathbb{IR}$ is the characteristic function
  of the cooperative game. We further assume that $w(\emptyset) = [0,0]$.

  The set of all interval cooperative games on a player set $N$ is denoted by $IG^N$.
\end{definition}

We often write $w(i)$ instead of $w(\{i\})$.

\begin{remark}
Each cooperative interval game in which the characteristic function maps to degenerate
intervals only can be associated with some classical cooperative game. The converse holds as well.
\end{remark}

\begin{definition} \emph{(Border games)}
  For every $(N,w) \in IG^N$, border games $(N,\underline{w}) \in G^N$ (lower border game) and $(N,\overline{w}) \in G^N$ (upper border game) are given by
  $\underline{w}(S) = \underline{w(S)}$ and $\overline{w}(S) = \overline{w(S)}$ for every $S \in 2^N$.
\end{definition}

\begin{definition} \emph{(Length game)} The length game of $(N,w) \in IG^N$ is the game $(N,|w|) \in G^N$ with
  $$|w|(S) = \overline{w}(S) - \underline{w}(S),\ \forall S \in 2^N \textrm{.}$$
\end{definition}

The basic notion of our approach will be a selection and consequently a selection imputation and a selection core.

\begin{definition} \emph{(Selection)} A game $(N,v) \in G^N$ is a selection of $(N,w)
\in IG^N$ if for every $S \in 2^N$ we have $v(S) \in w(S)$. Set of all selections
of $(N,w)$ is denoted by $\Sel(w)$.
\end{definition}

Note that border games are particular examples of selections.

\begin{definition} \emph{(Interval selection imputation)} The set of interval selection
imputations (or just selection imputations) of $(N,w) \in IG^N$ is defined as
$$\mathcal{SI}(w) = \bigcup \big\{I(v)\ |\ v \in \Sel(w) \big\} \textrm{.}$$
\end{definition}

\begin{definition} \emph{(Interval selection core)} The interval selection core (or just
selection core) of $(N,w) \in IG^N$ is defined as $$\mathcal{SC}(w) = \bigcup
\big\{C(v)\ |\ v \in \Sel(w) \big\} \textrm{.}$$ \end{definition}

Alparslan G\"{o}k \cite{gokphd} choose an approach using a weakly better
operator. That was inspired by \cite{partially}.

\begin{definition} \emph{(Weakly better operator $\succeq$)}
Interval $I$ is weakly better than interval $J$ ($I \succeq J$) if $\underline{I} \ge
\underline{J}$ and $\overline{I} \ge \overline{J}$. Furthermore, $I \preceq
J$ if and only if $\underline{I} \le \underline{J}$ and $\overline{I} \le
\overline{J}$. Interval $I$ is better than $J$ ($I \succ J$) if and only if $I
\succeq J$ and $I \not= J$.
\end{definition}

Their definition of imputation and core is as follows.

\begin{definition} \emph{(Interval imputation)} The set of interval imputations of
$(N,w) \in IG^N$ is defined as $$\mathcal{I}(w) := \Big\{ (I_1,I_2,\ldots,I_N)
\in \mathbb{IR}^N\ |\ \sum_{i\in N} I_i = w(N),\ I_i \succeq w(i),\ \forall i
\in N \Big\} \textrm{.}$$ \end{definition}

\begin{definition} \emph{(Interval core)} An interval core of $(N,w) \in IG^N$ is
defined as $$\mathcal{C}(w) := \Big\{ (I_1,I_2,\ldots,I_N) \in \mathcal{I}(w)\
|\ \sum_{i \in S}I_i \succeq w(S),\ \forall S \in 2^N \setminus \{\emptyset\}
\Big\} \textrm{.}$$ \end{definition}

An important difference between the definitions of interval and selection core and imputation is that selection concepts
yield payoff vectors from $\mathbb{R}^N$, while $\mathcal{I}$ and
$\mathcal{C}$ yield vectors from $\mathbb{IR}^N$.

 \emph{(Notation)}
Throughout the papers on cooperative interval games, notation, especially of
core and imputations, is not unified. It is, therefore, possible to encounter
different notation from ours.

Also, in these papers, the selection core is called the core of interval game. We
consider that confusing and that is why we use the term selection core instead.
The term selection imputation is used because of its connection with the selection
core.

The following classes of interval games have been studied earlier (see e.g. \cite{alparslan2009convex}).

\begin{definition} \emph{(Size monotonicity)} A game $(N,w) \in IG^N$ is size monotonic if for
every $T \subseteq S \subseteq N$ we have
$$|w|(T) \le |w|(S)\textrm{.}$$
That is, its length game is monotonic.

The class of size monotonic games on a player set $N$ is denoted by $\mathrm{SMIG}^N$.
\end{definition}

As we can see, size monotonic games capture situations in which an interval
uncertainty grows with the size of a coalition.

\begin{definition} \emph{(Superadditive interval game)} A game $(N,w) \in IG^N$ is a superadditive interval game if for
every $S,T \subseteq N$, $S \cap T = \emptyset$,
$$w(T) + w(S) \preceq w(S \cup T)\textrm{,}$$
and its length game is superadditive.
We denote by $\mathrm{SIG}^N$ the class of superadditive interval games on a player set $N$.
\end{definition}

We should be careful with the following analogy of a convex game since unlike for classical games, supermodularity
is not the same as convexity.

\begin{definition} \emph{(Supermodular interval game)} An interval game $(N,w)$ is supermodular interval
if for every $S \subseteq T \subseteq N$ holds
$$w(T) + v(S) \preceq w(S \cup T) + w(S \cap T)\textrm{.}$$
\end{definition}

We get immediately that an interval game is supermodular interval if and only if its
border games are convex.

\begin{definition} \emph{(Convex interval game)} An interval game $(N,w)$ is convex interval if its
border games and length game are convex.

We write $\mathrm{CIG}^N$ for a set of convex interval games on a player set $N$.
\end{definition}

A convex interval game is supermodular as well but the converse does not
hold in general.  See \cite{alparslan2009convex} for characterizations of convex interval games
and discussion on their properties.

\section{Selection-based classes of interval games}
\label{sec.43}

We will now introduce new classes of interval games based on the properties of
their selections. We think that it is a natural way to generalize special classes
from classical cooperative game theory. Consequently, we show their
characterizations and relation to classes from the preceding section.

\begin{definition} \emph{(Selection monotonic interval game)} An interval game
$(N,v)$ is selection monotonic if all its selections are monotonic games.
The class of such games on a player set $N$ is denoted by $\mathrm{SeMIG}^N$.
\end{definition}

\begin{definition} \emph{(Selection superadditive interval game)} An interval game
$(N,v)$ is selection superadditive if all its selections are superadditive
games. The class of such games on a player set $N$ is denoted by $\mathrm{SeSIG}^N$.
\end{definition}

\begin{definition} \emph{(Selection convex interval game)} An interval game $(N,v)$
is selection convex if all its selections are convex games. The class of such games on a player set $N$ is denoted
by $\mathrm{SeCIG}^N$.
\end{definition}

We see that many properties persist. For example, a selection convex game is a
selection superadditive as well. Selection monotonic and selection
superadditive are not subsets of each other but their intersection is
nonempty. Furthermore, the selection core of selection convex game is
nonempty, which is an easy observation.

We will now show characterizations of these three classes and consequently
show their relations to the existing classes presented in Section
\ref{coopgamesintro}.

\begin{theorem}
\label{chara}
  An interval game $(N,w)$ is selection monotonic if and only if for every $S,T \in 2^N$, $S \subsetneq T$
  $$ \overline{w}(S) \le \underline{w}(T)\textrm{.}$$
\end{theorem}
\begin{proof}
For the ``only if'' part, suppose that $(N,w)$ is a selection monotonic
and $\overline{w}(S) > \underline{w}(T)$ for some $S,T \in 2^N$, $S \subsetneq T$.
Then selection $(N,v)$ with $v(S) = \overline{w}(S)$ and $v(T) = \underline{w}(T)$ clearly
violates monotonicity and we arrive at a contradiction.

Now for the ``if'' part. For any two subsets $S,T$ of $N$, one of the situations $S \subsetneq T$, $T \subsetneq S$ or $S = T$ occurs.
For $S = T$, in every selection $v$, $v(S) \le v(S)$ holds. As for the other two situations, it is obvious
that monotonicity cannot be violated as well since $v(S) \le \overline{w}(S) \le \underline{w}(T) \le v(T)$.
\qed \end{proof}

Notice the importance of using $S \subsetneq T$ in the formulation of Theorem
\ref{chara}. That is because using of $S \subseteq T$ (thus allowing situation
$S = T$) would imply $\overline{w}(S) \le \underline{w}(S)$ for every $S$ in
selection monotonic game which is obviously not true in general. In
characterizations of selection superadditive and selection convex games, a
similar situation arises.

\begin{theorem} \label{supchar}
  An interval game $(N,w)$ is selection superadditive if and only if for every $S,T \in 2^N$ such that $S \cap T = \emptyset$, $S \not= \emptyset$, $T \not= \emptyset$
  $$ \overline{w}(S) + \overline{w}(T) \le \underline{w}(S \cup T)\textrm{.}$$
\end{theorem}
\begin{proof}
Similar to the proof of Theorem \ref{chara}.
\qed \end{proof}

We give a characterization of selection convex games as well:

\begin{theorem} \label{conchar}
  An interval game $(N,w)$ is selection convex if and only if for every $S,T \in 2^N$ such that $S \not\subseteq T$, $T \not\subseteq S$, $S \not= \emptyset$, $T \not= \emptyset$ holds
  $$ \overline{w}(S) + \overline{w}(T) \le \underline{w}(S \cup T) + \underline{w}(S \cap T)\textrm{.}$$
\end{theorem}
\begin{proof}
Similar to proof of Proposition \ref{chara}.
\end{proof}

Now let us look at a relation with existing classes of interval games.

For selection monotonic and size monotonic games, their relation is obvious.
For nontrivial games, i.e.\ games with the size of player set greater than
one, a selection monotonic game is not necessarily size monotonic and vice
versa.

\begin{theorem}
\label{superadditive}
For every player set $N$ with $|N| > 1$, the following assertions hold.
\begin{enumerate}[(i)]
  \item $\mathrm{SeSIG}^N \not\subseteq \mathrm{SIG}^N\textrm{.}$
  \item $\mathrm{SIG}^N \not\subseteq \mathrm{SeSIG}^N\textrm{.}$
  \item $\mathrm{SeSIG}^N \cap\, \mathrm{SIG}^N \not= \emptyset\textrm{.}$
\end{enumerate}
\end{theorem}

\begin{proof} 
In \emph{(i)},  we can construct the counterexample in the following way.

Let us construct game $(N,w)$.
For $w(\emptyset)$, the interval is given. Now for any nonempty coalition,
set $w(S) := [2|S|-2,2|S|-1]$. For any $S, T \in 2^N$ with $S$ and $T$ being nonempty and disjoint, the following holds
with the fact that $|S| + |T| = |S \cup T|$ taken into account.
\begin{align*}
\overline{w}(S) + \overline{w}(T) &= (2|S| - 1) + (2|T| - 1)\\
&= 2|S \cup T| - 2\\
&= \underline{w}(S \cup T)
\end{align*}
So $(N,w)$ is selection superadditive by Theorem \ref{supchar}. Its length game, however, is not
superadditive since for any two nonempty coalitions with empty
intersection $|w|(S) + |w|(T) = 2 \not\le 1 = |w|(S \cup T)$ holds.

In \emph{(ii)}, we can construct the following counterexample $(N, w')$. Set
$w'(S) = [0, |S|]$ for any nonempty $S$. The lower border game is trivially
superadditive. For the upper game, $\overline{w'}(S) +
\overline{w'}(T) = |S| + |T| = |S \cup T| = \overline{w'}(S \cup T)$ for any
$S, T$ with empty intersection, so the upper game is superadditive. Observe
that the length game is the same as the upper border game. This shows interval
superadditivity.

However, $(N, w')$ is clearly not selection superadditive because of nonzero
upper bounds, zero lower bounds of nonempty coalitions and the
characterization of $\mathrm{SeSIG}^N$ taken into account.

\emph{(iii)} Nonempty intersection can be argued easily by taking some superadditive game
$(N,c) \in G^N$. Then we can define corresponding game $(N,d) \in IG^N$ with $$d(S) =
[c(S),c(S)],\quad \forall S \in 2^N\textrm{.}$$ Game
$(N,d)$ is selection superadditive since its only selection is $(N,c)$. And it is
superadditive interval game since border games are supermodular and length game is
$|w|(S) = 0$ for every coalition, which trivially implies its superadditivity.
\qed \end{proof}

\begin{theorem}
\label{conv}
For every player set $N$ with $|N| > 1$, the following assertions hold.
\begin{enumerate}[(i)]
  \item $\mathrm{SeCIG}^N \not\subseteq \mathrm{CIG}^N\textrm{.}$
  \item $\mathrm{CIG}^N \not\subseteq \mathrm{SeCIG}^N\textrm{.}$
  \item $\mathrm{SeCIG}^N \cap\, \mathrm{CIG}^N \not= \emptyset\textrm{.}$
\end{enumerate}
\end{theorem}
\begin{proof}
For \emph{(i)}, take a game $(N,w)$ assigning to each nonempty coalition $S$ interval $[2^{|S|}-2, 2^{|S|}-1]$. 
From Theorem \ref{conchar}, we get that for inequalities which must hold in order to meet necessary conditions of game
to be selection convex, $|S| < |S \cup T|$ and $|T| < |S \cup T|$ must hold. That gives the following inequality:
\begin{align*}
\overline{w}(S) + \overline{w}(T) &\le (2^{|S \cup T| - 1} - 1) + (2^{|S \cup T| - 1} - 1)\\
&= 2^{|S \cup T|} - 2\\
&= \underline{w}(S \cup T)\\
&\le \underline{w}(S \cup T) + \underline{w}(S \cap T)
\end{align*}
This concludes that $(N,w)$ is selection convex. We see that the border games and the length game are convex too. To have a game
so that it is selection convex and not convex interval, we can take $(N, c)$ and set $c(S) := w(S)$ for $S \not= N$ and $c(N) := [\underline{w}(N),\underline{w}(N)]$.
Now the game $(N, c)$ is still selection convex, but its length game is not convex and $(N,v)$ is not a convex interval game, which is what we wanted.

In \emph{(ii)}, we can take a game $(N,w')$ from the proof of Theorem \ref{superadditive}\emph{(ii)}. From the fact that $|S| + |T| = |S \cup T| + |S \cap T|$,
it is clear that $\overline{w'}$ is convex. The lower border game is trivially convex and the length game is the same as upper. However,
for nonempty $S, T \in 2^N$ such that $S \not\subseteq T$, $T \not\subseteq S$, $S \not= \emptyset$, $T \not= \emptyset$, convex selection
games characterization is clearly violated.

As for \emph{(iii)}, we can use the same steps as in \emph{(iii)}
of Theorem \ref{superadditive} or we can use a game $(N,w)$ from \emph{(i)} of this theorem.
\qed \end{proof}

\section{Core coincidence}
\label{sec.4.core} \numberwithin{equation}{section}

In Alparslan-G\"{o}k's PhD thesis \cite{gokphd} and \cite{alparslan2011set}, the following question is suggested:
\begin{quote}
\emph{``A
difficult topic might be to analyze under which conditions the set of payoff
vectors generated by the interval core of a cooperative interval game
coincides with the core of the game in terms of selections of the interval
game.''}
\end{quote}

We decided to examine this topic. We call it the \emph{core coincidence
problem}. This section shows our results.

We remind the reader that whenever we talk about a relation and maximum,
minimum, maximal, minimal vectors, we mean the relation $\le$ on real vectors
unless we say otherwise.

The main thing to notice is that while the interval core gives us a set of
interval vectors, selection core gives us a set of real numbered vectors. To
be able to compare them, we need to assign to a set of interval vectors a set
of real vectors generated by these interval vectors. That is exactly what the
following function $\gen$ does.

\begin{definition} The function $\gen : 2^{\mathbb{IR}^N} \to 2^{\mathbb{R}^N}$ maps to every set of interval vectors a set of real vectors. It is defined as
$$\gen(S) = \bigcup_{s \in S} \big\{(x_1,x_2,\ldots,x_n)\ |\ x_i \in s_i\big\}\textrm{.}$$
\end{definition}

The core coincidence problem can be formulated as this: What are the necessary
and sufficient conditions to satisfy $\gen(\mathcal{C}(w)) = \mathcal{SC}(w)$?

The main results of this section are two theorems which can be seen as a partial
step towards an answer to the core coincidence problem.

In the following text by mixed system, we mean a system of equalities and
inequalities.

\begin{theorem} \label{temp}
For every interval game $(M,w)$ we have $\gen(\mathcal{C}(w)) \subseteq \mathcal{SC}(w)$.
\end{theorem}
\begin{proof}
  For any $x \in \gen(\mathcal{C}(w))$, the inequality $\underline{w}(N) \le \sum_{i \in N}x_i \le \overline{w}(N)$ obviously holds.
  Furthermore, $x$ is in the core for any selection of the interval game $(N, s)$ with $s$ given by
  $$s(S) = 
  \begin{cases}
    \big[\sum_{i \in N}x_i,\sum_{i \in N}x_i \big] \textrm{ if } S = N,\\
    \big[ \underline{w}(S), \min(\sum_{i \in S}x_i, \overline{w}(S)) \big] \textrm{ otherwise.}
  \end{cases}$$
  Clearly, $\Sel(s) \subseteq \Sel(w)$ and $\Sel(s) \not= \emptyset$. Therefore $\gen(\mathcal{C}(w)) \subseteq \mathcal{SC}(w)$.
\qed \end{proof}

\begin{theorem} \label{incidence} \emph{(Core coincidence characterization)}
For every interval game $(N,w)$ we have $\gen(\mathcal{C}(w)) = \mathcal{SC}(w)$ if and only if
for every $x \in \mathcal{SC}(w)$ there exist nonnegative vectors $l^{(x)}$ and $u^{(x)}$ such that
\begin{alignat}{3}
  &\sum_{i \in N}(x_i - l^{(x)}_i) &&= \underline{w}(N)\textrm{,}\\ 
  &\sum_{i \in N}(x_i + u^{(x)}_i) &&= \overline{w}(N)\textrm{,}\\
  &\sum_{i \in S}(x_i - l^{(x)}_i) &&\ge \underline{w}(S),\ \forall S \in 2^N \setminus \{\emptyset\}\textrm{,}\\ 
  &\sum_{i \in S}(x_i + u^{(x)}_i) &&\ge \overline{w}(S),\ \forall S \in 2^N \setminus \{\emptyset\}\textrm{.}
\end{alignat}
\end{theorem}
\begin{proof}
  First, we observe that with Theorem \ref{temp} taken into account, we only need
  to take care of $\gen(\mathcal{C}(w)) \supseteq \mathcal{SC}(w)$ to obtain equality.

  For $\gen(\mathcal{C}(w)) \supseteq \mathcal{SC}(w)$, suppose we have some $x \in \mathcal{SC}(w)$.
  For this vector, we need to find some interval $X \in \mathcal{C}(w)$ such that $x \in \gen({X})$. This is equivalent to
  the task of finding two nonnegative vectors $l^{(x)}$ and $u^{(x)}$ such that
  $$([x_1 - l^{(x)}_1,x_1 + u^{(x)}_1]),[x_2 - l^{(x)}_2, x_2 + u^{(x)}_2],\ldots ,[x_n - l^{(x)}_n, x_n + u^{(x)}_n]) \in \mathcal{C}(w)\textrm{.}$$
  From the definition of interval core, we can see that these two vectors have to satisfy exactly the mixed system $(4.1)-(4.4)$. That completes the proof.
\qed \end{proof}

\begin{example}
Consider an interval game with $N = \{1,2\}$ and $w(\{1\}) =
w(\{2\}) = [1,3]$ and $w(N) = [1,4]$. Then vector $(2,2)$ lies in the core of
the selection with $v(\{1\}) = v(\{2\}) = 2$ and $v(N) = 4$. However, to satisfy
equation $(4.1)$, we need to have $\sum_{i \in N} l_i = 3$ which means that
either $l_1$ or $l_2$ has to be greater than $1$. That means we cannot satisfy
$(4.3)$ and we conclude that $\gen(\mathcal{C}(w)) \not= \mathcal{SC}(w)$.
\end{example}

The following theorem shows that it suffices to check only minimal and maximal
vectors of $\mathcal{SC}(w)$.

\begin{theorem}
For every interval game $(N,w)$, if there exist vectors $q, r, x \in
\mathbb{R}^N$ such that $q,r \in \gen(\mathcal{C}(w))$ and  $q_i \le x_i \le r_i$ for every $i \in N$, then $x \in \gen(\mathcal{C}(w))$.
\end{theorem}
\begin{proof}
Let $l^{(r)}, u^{(r)}, l^{(q)}, u^{(q)}$ be the corresponding vectors in sense of Theorem \ref{incidence}.
We need to find vectors $l^{(x)}$ and $u^{(x)}$ satisfying $(4.1)-(4.4)$ of Theorem \ref{incidence}.

Let's define vectors $dq, dr \in \mathbb{R}^N$:
$$dq_i = x_i - q_i\textrm{,}$$
$$dr_i = r_i - x_i\textrm{.}$$

Finally, we define $l^{(x)}$ and $u^{(x)}$ in this way:
\begin{align*}
l^{(x)}_i &= dq_i + l^{(q)}_i\textrm{,} \\
u^{(x)}_i &= dr_i + u^{(r)}_i\textrm{.}
\end{align*}

We need to check that we satisfy $(4.1)-(4.4)$ for $x$, $l^{(x)}$ and $u^{(x)}$
We will show only $(4.2)$ since remaining ones can be done in a similar way.
\begin{align*}
  \sum_{i \in N}(x_i - l^{(x)}_i) &= \sum_{i \in N}(x_i - dq_i - l^{(q)}_i) \\
  &= \sum_{i \in N}(x_i - x_i + q_i - l^{(q)}_i) \\
  &= \sum_{i \in N}(q_i - l^{(q)}_i) \\
  &= \underline{w}(N)\textrm{.}
\end{align*}
\vskip -2em
\qed \end{proof}

For games with additive border games (see Definition \ref{add}) we get the
following result.

\begin{theorem} \label{check} For an interval game $(N,w)$ with additive border games, the payoff
vector $(\underline{w}(1),\underline{w}(2),\ldots,\underline{w}(n)) \in \gen(\mathcal{C}(w))$.
\end{theorem}
\begin{proof}
First, let us look at an arbitrary additive game $(A,v_A)$. From additivity condition and the fact that we can
write any subset of $A$ as a union of one-player sets we conclude that $v_A(A) = \bigcup_{i \in A} v_A(\{i \})$ for
every coalition $A$. This implies that vector $a$ with $a_i = v_A(\{i \})$ is in the core.

This argument can be applied to border games of $(N,w)$. The vector $q \in \mathbb{R}^N$ with $q_i = \underline{w}(i)$
is an element of the core of $(N,\underline{w})$ and an element of $\mathcal{SC}(w)$.

For the vector $q$ we want to satisfy the mixed system $(4.1)$-$(4.4)$ of Theorem \ref{incidence}.

Take the vector $l$ containing zeros only and the vector $u$ with $u_i = |w|(i)$. From the additivity, we get that $\sum_{i \in N}q_i - l_i = \underline{w}(N)$ and $\sum_{i \in N}q_i + u_i = \overline{w}(N)$.

Additivity further implies that inequalities $(4.3)$ and $(4.4)$ hold for $q$, $l$ and $u$.  Therefore, $q$ is an element of $\gen(\mathcal{C}(w))$.
\qed \end{proof}

Theorem \ref{check} implies that for games with additive border games, we
need to check the existence of vectors $l$ and $u$ from $(4.1)-(4.4)$ of Theorem \ref{incidence}
for maximal vectors of $\mathcal{SC}$ only. That follows from the fact that for any vector $y \in \mathcal{SC}(w)$ holds $(\underline{w}(1),\underline{w}(2),\ldots,\underline{w}(n)) \le y$. In other words, $(\underline{w}(1),\underline{w}(2),\ldots,\underline{w}(n))$ is a minimum vector of $\mathcal{SC}(w)$.

\section{Strong imputation and core}
\label{sec.4.strong}
In this section, our focus will be on a new concept of \emph{strong imputation} and \emph{strong core}.

\begin{definition} \emph{(Strong imputation)} For a game $(N,w) \in IG^N$ a strong imputation is a vector $x \in \mathbb{R}^N$
such that $x$ is an imputation for every selection of $(N,w)$.
\end{definition}

\begin{definition} \emph{(Strong core)} For a game $(N,w) \in IG^N$ the strong core is a set of vectors $x \in \mathbb{R}^N$
such that $x$ is an element of the core of every selection of $(N,w)$.
\end{definition}

Strong imputation and strong core can be considered as somewhat ``universal''
solutions. We show the following three simple facts about the strong core.

\begin{theorem} \label{dege} For every interval game with nonempty strong core, $w(N)$ is a degenerate interval.
\end{theorem}
\begin{proof} The theorem follows easily by the fact that an element $c$ of strong core must be efficient for every selection and
therefore $\sum_{i \in N} c_i = \underline{w}(N) = \overline{w}(N)$.
\qed \end{proof}

This leads us to a characterization of games with nonempty strong core.

\begin{theorem} \label{stchar} An interval game $(N,w)$ has a nonempty strong core if and only if $w(N)$ is a degenerate interval
and the upper game $\overline{w}$ has a nonempty core.
\end{theorem}
\begin{proof}
The theorem follows from a combination of Theorem \ref{dege} and the fact that an element $c$ of the strong core has to satisfy
$\sum_{i \in S} c_i \ge v(S),\ \forall v \in \Sel(w) ,\ \forall S \in 2^N \setminus \emptyset$. We see that this fact is equivalent to condition 
$\sum_{i \in S} c_i \ge \overline{w}(S),\ \forall S \in 2^N \setminus \emptyset$. Proving an equivalence is then straightforward.
\qed \end{proof}

We observe that we can easily derive a characterization of games with a
nonempty strong imputation set.

The strong core also has the following important property.

\begin{theorem}
For every element $c$ of the strong core of $(N,w)$, $c \in \gen(\mathcal{C}(w))$.
\end{theorem}
\begin{proof}
The vector $c$ has to satisfy mixed system $(4.1)$-$(4.4)$ of Theorem \ref{incidence} for some $l, u \in \mathbb{IR}^N$. We show that $l_i = u_i = 0$
will achieve this.

Equations $(4.1)$ and $(4.2)$ are satisfied by taking Theorem \ref{dege} into
account. Inequalities $(4.3)$ and $(4.4)$ are satisfied as the consequence of
Theorem \ref{stchar}.
\qed \end{proof}

The reason behind the using of name strong core and strong imputation comes
from interval linear algebra, where strong solutions of an interval system are
solutions for any realization (selection) of interval matrices $A$ and $b$ in
$Ax = b$.

One could ask why we do not introduce a strong game as a game in which each of
its selection has an nonempty core. This is because such games are
already defined as \emph{strongly balanced games} (see e.g.
\cite{alparslan2008cores}).

\section{Concluding remarks}

Selections of an interval game are very important since they do not contain
any additional uncertainty. On the top of that, selection-based classes and
the strong core and imputation have the crucial property that although we deal
with uncertain data, all possible outcomes preserve important properties. In
case of selection classes it is preserving superadditivity, supermodularity
etc. In case of the strong core it is an invariant of having particular stable
payoffs in each selection. Furthermore, ``weak'' concepts like $\mathcal{SC}$
are important as well since if $\mathcal{SC}$ is empty, no selection has a
stable payoff.

The importance of studying selection-based classes instead of the existing
classes using $\succeq$ operator can be further illustrated by the following
two facts:

\begin{itemize}[leftmargin=*]
\item Classes based on weakly better operator may contain games with selections that do not have any link with the defining property of their border games and consequently no link with the name of the class. For example, superadditive interval games may contain a selection that is not superadditive.
\item Selection-based classes are not contained in corresponding classes based on weakly better operator. Therefore, the results on existing classes are not directly extendable to selection-based classes.
\end{itemize}

Our results provide an important tool for handling cooperative situations
involving interval uncertainty which is a very common situation in various OR
problems. Some specific applications of interval games were already examined.
See \cite{alparslan2014cooperative,gok2009airport,forest} for
applications to mountain situations, airport games, and forest situations,
respectively. However, these papers do not use a selection-based approach and
therefore to study implications of our approach to them can be a theme for
future research.

To further study properties of selection-based classes is a possible topic.
One of the directions could be to introduce strictly selection convex games or
decomposable games and examine them. Another fruitful direction can be
extending of the definition of stable set to interval games using selections.
For example, one could look at a union or an intersection of stable sets for
each selection. Studying Shapley value and other concepts in interval
games context may be interesting as well. Some of these problems are work in
progress.

\section*{Acknowledgments}
This research was supported by the grant CE-ITI P202/12/G061 of GA\v{C}R.

\bibliographystyle{plain}
\bibliography{bibliography}

\end{document}